\documentclass[11pt,reqno]{amsart}
\usepackage[a4paper,margin=29mm]{geometry}
\usepackage[T1]{fontenc}
\usepackage[utf8]{inputenc}
\usepackage{lmodern}
\usepackage{amsmath,amssymb,mathtools}
\usepackage{microtype}
\usepackage[hidelinks]{hyperref}

\hypersetup{
  pdftitle={The Cartan determinant conjecture for representation-finite Artin algebras},
  pdfauthor={Liang Chen and Qianqian Guo},
  pdfsubject={Integral Cartan determinants of representation-finite Artin algebras}
}
\newtheorem{theorem}{Theorem}[section]
\numberwithin{equation}{section}
\DeclareMathOperator{\End}{End}
\DeclareMathOperator{\Hom}{Hom}

\DeclareMathOperator{\add}{add}
\DeclareMathOperator{\ind}{ind}
\DeclareMathOperator{\proj}{proj}
\DeclareMathOperator{\modcat}{mod}
\DeclareMathOperator{\rad}{rad}
\DeclareMathOperator{\soc}{soc}
\DeclareMathOperator{\pd}{pd}
\DeclareMathOperator{\gldim}{gl.dim}
\DeclareMathOperator{\gr}{gr}
\DeclareMathOperator{\id}{id}
\newcommand{\op}{\mathrm{op}}
\newcommand{\Z}{\mathbb Z}
\newcommand{\N}{\mathbb N}
\newcommand{\sHom}{\underline{\Hom}}
\newcommand{\cC}{\mathcal C}
\newcommand{\cD}{\mathcal D}
\newcommand{\cE}{\mathcal E}
\newcommand{\mult}[2]{[#1:#2]_{\oplus}}

\title[Cartan determinants of Artin algebras]{The Cartan determinant conjecture for representation-finite Artin algebras}
\author{Liang Chen}
\address{School of Mathematical Sciences, Capital Normal University, Beijing 100048, China}
\email{2210501003@cnu.edu.cn}
\author[Qianqian Guo]{Qianqian Guo$^*$}
\thanks{$^*$Corresponding author: Qianqian Guo.}
\address{School of Basic Education, Beijing Institute of Economics and Management, Beijing 100026, China}
\email{qianqianguo@biem.edu.cn}
\date{October 18, 2026}

\begin{document}

\begin{abstract}
We prove that every representation-finite Artin algebra of finite global dimension has integral Cartan determinant one.
\end{abstract}

\maketitle

\section{Introduction}

Let $A$ be an Artin algebra. Choose representatives $S_1,\ldots,S_n$ of the isomorphism classes of simple left $A$-modules, and let $P_i\twoheadrightarrow S_i$ be their projective covers. The integral Cartan matrix of $A$ is
\begin{equation}\label{eq:cartan-definition}
 C_A=\bigl([P_j:S_i]\bigr)_{1\leq i,j\leq n},
\end{equation}
where $[M:S_i]$ denotes composition multiplicity. Its determinant is independent of the simultaneous ordering of the simple modules and their projective covers.

Finite global dimension implies $\det C_A=\pm1$. Indeed, the alternating sums of finite projective resolutions of the simple modules provide an integral inverse of $C_A$. For finite-dimensional algebras over algebraically closed fields, this classical unimodularity result goes back to Eilenberg \cite[Proposition~21]{Eilenberg}. The Cartan determinant conjecture asks whether the sign is always positive.

Enomoto \cite[Theorem~A]{Enomoto} proved the conjecture for representation-finite finite-dimensional algebras over algebraically closed fields. His proof uses the standard form of a representation-finite algebra to obtain a positive grading from its mesh category. We establish the corresponding statement for Artin algebras.

\begin{theorem}\label{thm:main}
Let $A$ be a representation-finite Artin algebra. If $\gldim A<\infty$, then
\[
 \det C_A=1.
\]
\end{theorem}

The proof uses the radical grading of the Auslander algebra of $A$. Iyama's results on strict $\tau$-categories \cite{Iyama} show that the resulting category of projective modules can again be realized as a module category. The objects corresponding to the projective $A$-modules then determine a positively graded Artin algebra $B$ with $C_B=C_A$.

The essential homological step is to prove that $B$ has finite global dimension. We use the stable Auslander--Reiten combinatorics of Iyama, Nakaoka, and Palu \cite{INP} to recover stable Hom lengths from data preserved by the grading. A local formula involving a left $\tau$-sequence then recovers the nonprojective direct summands of each syzygy. This allows induction on projective dimension. The proof concludes with the graded Cartan argument, whose Artinian form is recorded in \cite[Corollary~1.9]{Saorin}.

\subsection*{Conventions and notation}
All algebras are nonzero, associative, and unital. An Artin algebra is an algebra finite as a module over a commutative Artin ring $K$. Unless stated otherwise, modules are finitely generated left modules. We write $\modcat A$ for their category and $\proj A$ for its full subcategory of projective modules. The notation $\ind\cC$ denotes a set of representatives of the indecomposable objects of a Krull--Schmidt category $\cC$. For an object $G$, the category $\add G$ consists of the direct summands of finite direct sums of copies of $G$. Thus $A$ is representation-finite precisely when $\modcat A=\add G$ for some $G$.

We write $J(R)$ for the Jacobson radical of a ring $R$, $\ell_K V$ for the $K$-length of a finite $K$-module, and $\ell_A M$ for composition length as an $A$-module. For an indecomposable object $X$, the notation $\mult{M}{X}$ denotes its direct-summand multiplicity in $M$; it is distinct from the composition multiplicity in \eqref{eq:cartan-definition}. We use minimal projective resolutions and write $\Omega_A M$ for the kernel of a projective cover. Composition is written from right to left, so that $gf=g\circ f$.

\section{Auslander--Reiten data and radical layers}\label{sec:AR}

We recall the notation needed to apply the categorical results in \cite{Iyama,INP}. These results will be used only for categories with finitely many indecomposable objects and Hom spaces of finite $K$-length.

For a Krull--Schmidt category $\cC$, its radical is the ideal given by
\[
 \rad_{\cC}(X,Y)
 =\{f:X\to Y\mid 1_X-gf\text{ is invertible for every }g:Y\to X\}.
\]
Its powers are formed by finite sums of composites. If $\cC=\add G$, these powers are the corresponding components of the powers of $J(\End_{\cC}(G))$. In particular, they vanish in sufficiently high degree in the situation considered here. For an indecomposable object $X$, put
\begin{equation}\label{eq:weights}
 \Delta_X=\End_{\cC}(X)/J(\End_{\cC}(X)),
 \qquad w_X=\ell_K\Delta_X.
\end{equation}
The ring $\Delta_X$ is a division ring and $w_X$ is a positive integer.

We use the notation
\begin{equation}\label{eq:tau-sequences}
 \tau^+X\longrightarrow\theta^+X\longrightarrow X,
 \qquad
 X\longrightarrow\theta^-X\longrightarrow\tau^-X
\end{equation}
for the right and left $\tau$-sequences of \cite[Section~1.3]{Iyama}. The first map in a right $\tau$-sequence is a minimal weak kernel of the second; the dual condition holds for a left $\tau$-sequence. The maps have the almost split factorization properties on their respective sides. Strictness means that the first map of a right $\tau$-sequence is a monomorphism and the second map of a left $\tau$-sequence is an epimorphism, in the category in question. In particular, a strict left $\tau$-sequence gives, for every $N$, an exact sequence
\begin{equation}\label{eq:left-tau-hom}
 0\longrightarrow\Hom_{\cC}(\tau^-X,N)
 \longrightarrow\Hom_{\cC}(\theta^-X,N)
 \longrightarrow\rad_{\cC}(X,N)\longrightarrow0.
\end{equation}

For an Artin algebra $R$, the category $\modcat R$ is a strict $\tau$-category \cite[Section~1.3.2(1)]{Iyama}. At nonprojective and noninjective vertices, \eqref{eq:tau-sequences} comes from almost split short exact sequences. At the boundary, the sequences are
\begin{equation}\label{eq:boundary}
 0\longrightarrow J(R)P\longrightarrow P,
 \qquad
 I\longrightarrow I/\soc I\longrightarrow0
\end{equation}
for indecomposable projective $P$ and indecomposable injective $I$. Consequently,
\begin{equation}\label{eq:boundary-identification}
 \tau^+X=0\iff X\text{ is projective},
 \qquad
 \tau^-X=0\iff X\text{ is injective}.
\end{equation}

The functions $\theta^\pm$ and $\tau^\pm$ extend additively to the free abelian group on $\ind\cC$. In Iyama's terminology, a positive integer-valued function $l$ on $\ind\cC$ is additive if the defects
\[
 l(X)-l(\theta^\pm X)+l(\tau^\pm X)
\]
are nonnegative and vanish whenever $\tau^\pm X\ne0$. For $\cC=\modcat R$, the function $l(X)=\ell_R X$ satisfies
\begin{equation}\label{eq:length-defects}
 \begin{aligned}
 l(X)-l(\theta^+X)+l(\tau^+X)
 &=\begin{cases}1,&X\text{ is projective},\\0,&X\text{ is not projective},\end{cases}\\
 l(X)-l(\theta^-X)+l(\tau^-X)
 &=\begin{cases}1,&X\text{ is injective},\\0,&X\text{ is not injective}.
 \end{cases}
 \end{aligned}
\end{equation}
Here $X$ is indecomposable. The boundary values follow from \eqref{eq:boundary}, since the top of $P$ and the socle of $I$ are simple. The other values follow from additivity of length on almost split short exact sequences.

For an Artin algebra $R$, its stable category $\underline{\modcat}R$ has morphism spaces
\[
 \sHom_R(X,Y)=\Hom_R(X,Y)/\mathcal P_R(X,Y),
\]
where $\mathcal P_R(X,Y)$ consists of maps factoring through projective modules. Its indecomposable nonzero objects are represented by the nonprojective indecomposable $R$-modules. We shall use \cite[Theorem~7.15(2)]{INP}: this stable category is a $\tau$-category, and its valued $\tau$-quiver is determined by the valued Auslander--Reiten translation quiver of $\modcat R$, including its projective and injective vertices.

Here is the length calculation that will be needed. Write $\cE=\underline{\modcat}R$, and denote its right $\tau$-sequences by $\tau T\to\theta T\to T$. Let $K_0(\cE)=\Z[\ind\cE]$ be the split Grothendieck group, with positive monoid $K_0(\cE)^+=\N[\ind\cE]$. For $z\in K_0(\cE)$, let $z_+$ denote its positive part. Define maps $\Theta_r:K_0(\cE)^+\to K_0(\cE)^+$ by the recursion
\begin{equation}\label{eq:theta-recursion}
 \Theta_0=\id,\qquad \Theta_1=\theta,\qquad
 \Theta_r(T)=\bigl(\theta\Theta_{r-1}(T)-\tau\Theta_{r-2}(T)\bigr)_+
 \quad(r\geq2).
\end{equation}
These functions depend only on the valued $\tau$-quiver \cite[Definition~7.6]{INP}. If $R$ is representation-finite, the radical of $\cE$ is nilpotent, and \cite[Corollary~7.21(3)]{INP} gives
\begin{equation}\label{eq:stable-length}
 \operatorname{length}\bigl(\cE(Y,M)_{\End_{\cE}(Y)}\bigr)
 =\sum_{r\geq0}\mult{\Theta_r(M)}{Y}
 \qquad(Y\in\ind\cE).
\end{equation}
The length on the left is the length of the right endomorphism-ring module given by precomposition. The sum is finite. All the categorical results just recalled apply to Artin algebras.

\section{Proof of the theorem}

\begin{proof}[Proof of Theorem~\ref{thm:main}]
Fix a commutative Artin base ring $K$ for $A$, and put $d=\gldim A$. A Morita equivalence preserves representation-finiteness, projective dimensions, projective covers, and composition multiplicities. We may therefore assume that $A$ is basic. Let
\[
 G=\bigoplus_{X\in\ind(\modcat A)}X,
 \qquad \Gamma=\End_A(G)^{\op}.
\]
The functor $\Hom_A(G,-)$ identifies $\modcat A=\add G$ with $\proj\Gamma$. For $X\in\ind(\modcat A)$, let $\varepsilon_X\in\Gamma$ be the primitive idempotent corresponding to the summand $X$ of $G$, and set
\[
 Q_X=\Hom_A(G,X)\simeq\Gamma\varepsilon_X.
\]

\smallskip
\noindent\emph{Step 1: Radical grading and an abelian realization.}
Put
\[
 I=J(\Gamma),\qquad
 \Lambda=\gr_I\Gamma=\bigoplus_{r\geq0}I^r/I^{r+1},
 \qquad \cD=\proj\Lambda.
\]
We use the same symbol $\varepsilon_X$ for its degree-zero image in $\Lambda_0=\Gamma/I$. Since $\proj\Gamma$ is equivalent to the strict $\tau$-category $\modcat A$, the algebra $\Gamma$ is a strict $\tau$-algebra. By \cite[Section~2.5]{Iyama}, the same holds for $\Lambda$.

We explain which numerical data are preserved. The right $\tau$-sequence ending at $X$ gives a minimal projective resolution of the corresponding simple $\Gamma$-module $L_X$:
\begin{equation}\label{eq:simple-Gamma-resolution}
 0\longrightarrow Q_{\tau^+X}
 \longrightarrow Q_{\theta^+X}
 \longrightarrow Q_X\longrightarrow L_X\longrightarrow0,
\end{equation}
where zero terms are allowed. Iyama's radical-layer statement \cite[Lemma~2.5.1]{Iyama} gives exact sequences
\begin{equation}\label{eq:Gamma-layers}
 0\longrightarrow I^{r-1}Q_{\tau^+X}
 \longrightarrow I^rQ_{\theta^+X}
 \longrightarrow I^{r+1}Q_X\longrightarrow0
 \qquad(r\geq0),
\end{equation}
with $I^{-1}Q=Q$. Use the shift convention $V\langle a\rangle_r=V_{r-a}$. Taking the quotients of consecutive sequences gives the graded projective resolution
\begin{equation}\label{eq:graded-simple-resolution}
 \begin{aligned}
 0\longrightarrow(\gr_I Q_{\tau^+X})\langle2\rangle
 &\longrightarrow(\gr_I Q_{\theta^+X})\langle1\rangle\\
 &\longrightarrow\gr_I Q_X\longrightarrow L_X\longrightarrow0.
 \end{aligned}
\end{equation}
Since $\Lambda_0=\Gamma/I$ is semisimple and $\Lambda_{>0}$ is nilpotent, $J(\Lambda)=\Lambda_{>0}$. The maps in \eqref{eq:graded-simple-resolution} have images in the radicals of their targets, so this resolution remains minimal after forgetting the grading. Thus the correspondence
\[
 X\longleftrightarrow Q_X^\sharp:=\Lambda\varepsilon_X
\]
preserves the objects $\theta^+X$ and $\tau^+X$, including their indecomposable multiplicities. Since $\gr_{J(\Gamma^{\op})}\Gamma^{\op}\simeq\Lambda^{\op}$, projective duality $\Hom_\Gamma(-,\Gamma)$ allows the same argument for $\Gamma^{\op}$ to be applied to the left $\tau$-sequences. Thus $\theta^-X$ and $\tau^-X$ are preserved as well. In particular, both boundary sets in \eqref{eq:boundary-identification} and the valuations of the Auslander--Reiten quiver are preserved.

Transfer the integers $l(X)=\ell_A X$ to the objects $Q_X^\sharp$. By \eqref{eq:length-defects}, the transferred function is additive, and its left defect is strictly positive at every vertex with $\tau^-=0$. In the notation of \cite[Section~1.3.1]{Iyama}, this says
\[
 l^-=\ind^-_1\cD.
\]
Iyama's realization criterion \cite[Theorem~3.3(5)$\Rightarrow$(3)]{Iyama} now gives an equivalence
\begin{equation}\label{eq:abstract-realization}
 \cD\simeq\modcat B'
\end{equation}
for an Artin algebra $B'$. In particular, $\cD$ is an abelian category, with the abelian structure transported from $\modcat B'$. This is the point where we use the full length function of $\modcat A$, in addition to strictness of the graded $\tau$-category.

\smallskip
\noindent\emph{Step 2: The projective corner and its $K$-linear realization.}
Let
\[
 e=\sum_{i=1}^n\varepsilon_{P_i},
 \qquad Q=\Lambda e,
 \qquad B=e\Lambda e.
\]
The objects $Q_{P_i}^\sharp$ are exactly the indecomposable objects of $\cD$ with $\tau^+=0$. Under \eqref{eq:abstract-realization}, these become the indecomposable projective $B'$-modules. Hence $Q$ becomes a projective generator. Morita theory therefore shows that the canonical Yoneda functor
\begin{equation}\label{eq:Yoneda}
 H=\Hom_{\cD}(Q,-):\cD\xrightarrow{\ \sim\ }
 \modcat\End_{\cD}(Q)^{\op}=\modcat B
\end{equation}
is an equivalence. Define
\begin{equation}\label{eq:sharp-objects}
 X^\sharp=H(Q_X^\sharp)\simeq e\Lambda\varepsilon_X,
\end{equation}
and extend this notation to finite direct sums. These objects exhaust the indecomposable $B$-modules. In particular, $B$ is representation-finite, and
\begin{equation}\label{eq:projectives-sharp}
 P_i^\sharp=B\varepsilon_{P_i}
\end{equation}
are its indecomposable projective modules. The correspondence $X\leftrightarrow X^\sharp$ preserves the two-sided $\tau$-data. It is an object correspondence; an equivalence between $\modcat A$ and $\modcat B$ is not asserted.

We keep track of the base ring explicitly. The ideal $J(K)\Gamma$ is nilpotent, so $J(K)\Gamma\subseteq I$. More precisely, $J(K)I^r\subseteq I^{r+1}$ for every $r\geq0$, so $J(K)$ annihilates each homogeneous layer $I^r/I^{r+1}$. Consequently, the natural central $K$-action on $\Lambda$ factors through $K/J(K)$, and on each homogeneous layer it is the original action on $I^r/I^{r+1}$. Thus $\Lambda$ and $B$ are Artin $K$-algebras. The functor in \eqref{eq:Yoneda} is $K$-linear, independently of any choice of linear structure on \eqref{eq:abstract-realization}.

For indecomposable $X,Y$, full faithfulness in \eqref{eq:Yoneda} gives $K$-module isomorphisms
\begin{equation}\label{eq:Hom-layers}
 \begin{aligned}
 \Hom_B(X^\sharp,Y^\sharp)
 &\simeq\Hom_\Lambda(\Lambda\varepsilon_X,\Lambda\varepsilon_Y)\\
 &\simeq\varepsilon_X\Lambda\varepsilon_Y
 =\bigoplus_{r\geq0}
 \frac{\varepsilon_X I^r\varepsilon_Y}{\varepsilon_X I^{r+1}\varepsilon_Y}.
 \end{aligned}
\end{equation}
On the other hand, $\varepsilon_X\Gamma\varepsilon_Y\simeq\Hom_A(X,Y)$ as $K$-modules. Additivity of length along the finite radical filtration yields
\begin{equation}\label{eq:Hom-length-preserved}
 \ell_K\Hom_B(X^\sharp,Y^\sharp)=\ell_K\Hom_A(X,Y).
\end{equation}
The degree-zero quotient also gives, as $K$-algebras,
\begin{equation}\label{eq:division-preserved}
 \Delta_{X^\sharp}\simeq\Delta_X,
 \qquad w_{X^\sharp}=w_X.
\end{equation}
Indeed, $\Gamma/I\simeq\prod_X\Delta_X^{\op}$, and the same division factors occur in $\Lambda/J(\Lambda)=\Lambda_0$. Since $\End_\Lambda(\Lambda\varepsilon_X)\simeq(\varepsilon_X\Lambda\varepsilon_X)^{\op}$, its residue division ring is $\Delta_X$. These comparisons use lengths over $K$, rather than dimensions over a splitting field.

\smallskip
\noindent\emph{Step 3: Preservation of stable Hom lengths.}
Both $\modcat A$ and $\modcat B$ have almost split sequences, enough projectives, and source morphisms. Their valued Auslander--Reiten translation quivers, including their projective and injective vertices, agree by Steps~1 and~2. Therefore \cite[Theorem~7.15(2)]{INP} identifies the valued $\tau$-data of their stable categories. This invokes the full stable-category construction: when an almost split sequence ending at $X$ has projective middle term, its image has right $\tau$-sequence $0\to0\to X$, as in \cite[Lemma~7.17(2)]{INP}.

For a nonprojective indecomposable $Y$, an endomorphism of $Y$ factoring through projectives cannot be invertible; otherwise $Y$ would be a direct summand of a projective. Such endomorphisms consequently lie in $J(\End_A(Y))$. Thus $\End_{\underline{\modcat}A}(Y)$ has residue division ring $\Delta_Y$, and the analogous assertion holds for $B$. Every composition factor of a finite module over either of these local endomorphism rings has $K$-length $w_Y$. Multiplying \eqref{eq:stable-length} by $w_Y$, and using the matching recursions \eqref{eq:theta-recursion} together with \eqref{eq:division-preserved}, gives
\begin{equation}\label{eq:stable-Hom-preserved}
 \ell_K\sHom_A(Y,M)
 =\ell_K\sHom_B(Y^\sharp,M^\sharp)
 \qquad(Y,M\in\modcat A).
\end{equation}
For projective $Y$ both sides are zero, and the general case follows by direct sums. Nilpotence of the stable radicals ensures that all the length sums used here are finite.

\smallskip
\noindent\emph{Step 4: Recovery of nonprojective syzygies.}
For an indecomposable $X$ and an integer-valued function $f$ on objects that is additive under finite direct sums, put
\[
 \delta_X f=f(X)-f(\theta^-X)+f(\tau^-X).
\]
Write $h_N(Y)=\ell_K\Hom_A(Y,N)$. Since
\[
 \ell_K\bigl(\Hom_A(X,N)/\rad_{\modcat A}(X,N)\bigr)
 =w_X\mult{N}{X},
\]
the exact sequence \eqref{eq:left-tau-hom} gives
\begin{equation}\label{eq:local-multiplicity}
 \delta_Xh_N=w_X\mult{N}{X}.
\end{equation}
The same calculation holds over $B$.

Take a projective cover $P_M\twoheadrightarrow M$. Applying $\Hom_A(Y,-)$ to its kernel sequence gives
\begin{equation}\label{eq:cover-stable-sequence}
 \begin{aligned}
 0\longrightarrow\Hom_A(Y,\Omega_A M)
 &\longrightarrow\Hom_A(Y,P_M)\\
 &\longrightarrow\Hom_A(Y,M)
 \longrightarrow\sHom_A(Y,M)\longrightarrow0.
 \end{aligned}
\end{equation}
For the last identification, every map from a projective module to $M$ lifts through $P_M\twoheadrightarrow M$. Hence the image of the middle map is exactly the $K$-submodule of maps factoring through projectives.

Set $u_M(Y)=\ell_K\sHom_A(Y,M)$. Taking lengths in \eqref{eq:cover-stable-sequence} yields
\[
 h_{\Omega_A M}-h_{P_M}=u_M-h_M.
\]
Apply $\delta_X$ and use \eqref{eq:local-multiplicity}. If $X$ is nonprojective, then $\mult{P_M}{X}=0$, and we obtain
\begin{equation}\label{eq:syzygy-formula}
 \mult{\Omega_A M}{X}
 =-\mult{M}{X}
 +\frac{u_M(X)-u_M(\theta^-X)+u_M(\tau^-X)}{w_X}.
\end{equation}
The right side is preserved by \eqref{eq:division-preserved}, \eqref{eq:stable-Hom-preserved}, and preservation of the left $\tau$-data. Comparing with the same formula over $B$ gives
\begin{equation}\label{eq:syzygy-preserved}
 \mult{\Omega_A M}{X}
 =\mult{\Omega_B M^\sharp}{X^\sharp}
 \qquad(X\text{ nonprojective}).
\end{equation}

We now induct on $\pd_A M$ to obtain
\begin{equation}\label{eq:pd-preserved}
 \pd_B M^\sharp=\pd_A M
 \qquad(0\ne M\in\modcat A).
\end{equation}
It suffices to treat indecomposable $M$, since projective dimension takes the maximum on a finite direct sum. The case $\pd_A M=0$ follows from the correspondence of projective objects. If $\pd_A M=1$, then \eqref{eq:syzygy-preserved} implies that $\Omega_B M^\sharp$ is projective. It is nonzero, since $M^\sharp$ is nonprojective, and consequently $\pd_B M^\sharp=1$. Suppose now that $\pd_A M=m\geq2$. Then $\pd_A\Omega_A M=m-1\geq1$. Thus the nonprojective indecomposable summands of $\Omega_A M$ form a nonempty collection, and the maximum of their projective dimensions is $m-1$. The induction hypothesis and \eqref{eq:syzygy-preserved} give the same maximum for the nonprojective summands of $\Omega_B M^\sharp$. Projective summands do not affect this maximum. Hence $\pd_B\Omega_B M^\sharp=m-1$ and $\pd_B M^\sharp=m$, proving \eqref{eq:pd-preserved}.

Every simple $B$-module is one of the modules $M^\sharp$, and hence has projective dimension at most $d$. This bounds the global dimension on all modules as well. Indeed, every $B$-module has a finite radical filtration whose factors are semisimple, and arbitrary direct sums of simple modules of projective dimension at most $d$ have the same bound. Successive extensions preserve it. We conclude that
\begin{equation}\label{eq:B-finite-gldim}
 \gldim B\leq d<\infty.
\end{equation}

\smallskip
\noindent\emph{Step 5: Equality of the integral Cartan matrices.}
Let $T_i$ be the simple top of $P_i^\sharp$. Exactness of $\Hom_A(P_i,-)$ and a composition series give
\begin{equation}\label{eq:Cartan-length}
 \ell_K\Hom_A(P_i,P_j)
 =[P_j:S_i]\,\ell_K\End_A(S_i).
\end{equation}
In fact, $\Hom_A(P_i,S_l)=0$ for $l\ne i$, while $\Hom_A(P_i,S_i)\simeq\End_A(S_i)$. The natural map $\End_A(P_i)\to\End_A(S_i)$ is surjective by projectivity. Its kernel is $J(\End_A(P_i))$, since $\End_A(P_i)$ is local and $\End_A(S_i)$ is a division ring. Thus
\[
 \ell_K\End_A(S_i)=w_{P_i}.
\]
The analogous assertions hold for $P_i^\sharp$ and $T_i$. Using \eqref{eq:Hom-length-preserved} and \eqref{eq:division-preserved} in \eqref{eq:Cartan-length} now gives
\[
 [P_j:S_i]
 =\frac{\ell_K\Hom_A(P_i,P_j)}{w_{P_i}}
 =\frac{\ell_K\Hom_B(P_i^\sharp,P_j^\sharp)}{w_{P_i^\sharp}}
 =[P_j^\sharp:T_i].
\]
Therefore
\begin{equation}\label{eq:Cartan-preserved}
 C_A=C_B.
\end{equation}

\smallskip
\noindent\emph{Step 6: The graded Cartan calculation.}
The grading of $\Lambda$ induces
\[
 B=\bigoplus_{r\geq0}B_r,
 \qquad B_r=e(I^r/I^{r+1})e.
\]
The algebra $B_0=e(\Gamma/I)e$ is semisimple. Since the grading is finite, $B_{>0}$ is nilpotent, and hence
\[
 J(B)=B_{>0},\qquad B_0\simeq B/J(B).
\]
The idempotents $\varepsilon_{P_i}$ lie in degree zero. Give $P_i^\sharp=B\varepsilon_{P_i}$ its induced grading and put $T_i$ in degree zero. Then $T_i\simeq B_0\varepsilon_{P_i}$.

For a finite graded $B$-module $V$, define the column vector
\[
 h(V;t)_i=\sum_{r\in\Z}[V_r:T_i]_{B_0}\,t^r.
\]
It is additive on short exact sequences of graded modules with degree-zero maps. The graded Cartan matrix is
\begin{equation}\label{eq:graded-Cartan}
 C_B(t)_{ij}=h(P_j^\sharp;t)_i
 =\sum_{r\geq0}[(P_j^\sharp)_r:T_i]_{B_0}\,t^r.
\end{equation}
It belongs to $\operatorname{Mat}_n(\Z[t])$ and satisfies
\begin{equation}\label{eq:graded-specializations}
 C_B(0)=I_n,\qquad C_B(1)=C_B.
\end{equation}
For the second equality, restrict a $B$-composition series to $B_0$. Every simple $B$-module remains the corresponding simple $B_0$-module, so the multiplicities add as asserted.

For completeness, we explain the existence of finite graded projective resolutions in this Artinian setting. If $V$ is a finite graded $B$-module, let $\overline V=V/J(B)V$. Semisimplicity of $B_0$ provides a degree-preserving $B_0$-linear section $s:\overline V\to V$. The map
\[
 B\otimes_{B_0}\overline V\longrightarrow V,
 \qquad b\otimes v\longmapsto b\,s(v),
\]
is a graded projective cover: its source is a finite direct sum of graded shifts of the $P_i^\sharp$, and the induced map on tops is an isomorphism. Nakayama's lemma makes the map surjective, and the isomorphism on tops puts its kernel in the radical of its source. Thus its underlying map is also an ordinary projective cover.

Iterating gives a graded resolution whose underlying resolution is minimal. It terminates by \eqref{eq:B-finite-gldim}. Indeed, once a syzygy is projective, its projective cover is an isomorphism: the cover splits, and a direct summand contained in the radical of a finite projective module must be zero by Nakayama's lemma.

In a resolution of $T_j$, only nonnegative shifts occur, since $T_j$ starts in degree zero and each subsequent kernel lies in the positive-degree radical of the preceding projective term. Write its finite graded projective resolution as
\[
 0\longrightarrow R_j^{d_j}\longrightarrow\cdots
 \longrightarrow R_j^0\longrightarrow T_j\longrightarrow0,
\]
with
\[
 R_j^q\simeq\bigoplus_{i=1}^n\bigoplus_{a\geq0}
 \bigl(P_i^\sharp\langle a\rangle\bigr)^{\oplus m_{i,a}^{q,j}}.
\]
Define
\[
 E(t)_{ij}=\sum_{q=0}^{d_j}\sum_{a\geq0}(-1)^q m_{i,a}^{q,j}t^a.
\]
All these sums are finite. Taking alternating sums of graded multiplicity vectors gives
\[
 C_B(t)E(t)=I_n.
\]
Thus $\det C_B(t)$ is a unit of $\Z[t]$, and its value at $t=0$ is one. It follows that $\det C_B(t)=1$. Evaluating at $t=1$ and applying \eqref{eq:Cartan-preserved} proves
\[
 \det C_A=\det C_B=1.
\]
\end{proof}

\bigskip
\noindent\textbf{Acknowledgements.} ChatGPT assisted with this work.

\end{document}